\documentclass[12pt]{article}
\usepackage{amssymb}

\usepackage{amsmath,amsthm}
\providecommand{\U}[1]{\protect\rule{.1in}{.1in}}

\makeatletter
\def\@seccntformat#1{\csname the#1\endcsname.\quad}
\makeatother

\newtheorem{theorem}{Theorem}

\newtheorem{corollary}{Corollary}

\newtheorem{definition}{\noindent Definition}
\newtheorem{example}{Example}

\newtheorem{lemma}{Lemma}
\newtheorem{notation}{Note}

\newtheorem{remark}{Remark}

\renewenvironment{proof}[1][Proof]{\noindent\textbf{#1.} }{\ \rule{0.5em}{0.5em}}

\begin{document}

\title{{\Large \textbf{Non-Archimedean Scale Invariance and Cantor Sets}}}
\author{Santanu Raut\thanks{{Chhatgurihati Seva Bhavan Sikshayatan High Scool, 
New Town, Coochbihar-736101, {\it email}}: raut\_santau@yahoo.com} \ and Dhurjati Prasad Datta\thanks{corresponding author, 
\textit{email}: dp\_datta@yahoo.com} \and \textit{Department of Mathematics}
\and \textit{University of North Bengal, Siliguri,West Bengal, India, Pin}
734013}
\date{}
\maketitle

\begin{abstract}
The framework of a new scale invariant analysis on a Cantor set $C\subset $ $%
I=[0,1]\ $, presented originally in {\it S. Raut and D. P. Datta, Fractals, 
17, 45-52, (2009)}, is clarified and extended further. For an arbitrarily small $\varepsilon >0$, elements $\tilde{x}$ in $I\backslash C$ satisfying $0<\tilde{x}<\varepsilon <x,\ x\in
C $ \ together with an inversion rule are called relative infinitesimals
relative to the scale $\varepsilon$. A non-archimedean absolute value $v(%
\tilde{x})=\log _{\varepsilon ^{-1}}\frac{\varepsilon }{\tilde{x}}, \ \varepsilon \rightarrow 0$ is
assigned to each such infinitesimal which is then shown to induce a
non-archimedean structure in the full Cantor set $C$. A valued measure 
constructed using the new absolute value is shown to give rise to the
finite Hausdorff measure of the set. The definition of differentiability on $%
C$ in the non-archimedean sense is introduced. The associated Cantor
function is shown to relate to the valuation on $C$ which is then
reinterpreated as a locally constant function in the extended
non-archimedean space. The definitions and the constructions are verified
explicitly on a  Cantor set which is defined recursively from $I$
deleting $q$ number of open intervals each of length $\frac{1}{r}$ leaving\
out $p$ numbers of closed intervals so that $p+q=r.$
\end{abstract}

\begin{center} 
{\bf Key Words}: Non-archimedean, scale invariance, Cantor set, Cantor function
\end{center}

\begin{center} 
{\bf AMS Classification Numbers:} 26E30, 26E35, 28A80
\end{center}

\begin{center}
{\large To appear in FRACTALS, March 2010}
\end{center}

\baselineskip=19pt

\newpage
\section{\hspace{-2ex}Introduction}

\indent\indent A Cantor set is a totally disconnected compact and perfect subset of the
real line. Such a set displays many paradoxical properties. Although the set
is uncountable, it's Lebesgue measure vanishes. The topological dimension of
the set is also zero. Cantor set is an example of a self-similar fractal set
that arises in various fields of applications. The chaotic attractors of a
number of one dimensional maps; such as the logistic maps, turn out to be
topologically equivalent to Cantor sets. Cantor set also arises in
electrical communications \cite{bma1}, in biological systems \cite{bw2}, and
diffusion processes \cite{mtb3}. Recently there have been a lot of interest
in developing a framework of analysis on a Cantor like fractal sets \cite%
{jki4,rda5}. Because of the disconnected nature, methods of ordinary real
analysis break down on a Cantor set. Various approaches based on the
fractional derivatives \cite{rl7,ka8} and the measure theoretic harmonic
analysis \cite{um6} have already been considered at length in the
literature. However, a simpler intuitively appealing approach is still
considered to be welcome.

Recently, we have been developing a non-archimedean framework \cite{na9} of a scale
invariant analysis which will be naturally relevant on a Cantor set \cite
{sd10,dpd11}. For definiteness, we consider the Cantor set $C$ to be a subset
of the unit interval $I$ $=[0,1]$. We introduce a non-archimedean absolute
value on $C\ $exploiting a concept of relative infinitesimals which
correspond to the arbitrarily small elements $\tilde x$ of $I\backslash C$
satisfying $0<\tilde x<\varepsilon ,$ $\varepsilon \rightarrow 0^{+}$ $($together
with an inversion rule$)$ relative to the scale $\varepsilon .$\ In
ref. \cite{sd10}, we have presented the details of the construction in
the light of the middle third Cantor set. Here, we develop the formalism
afresh, bringing in more {\em clarity} in the approach initiated in ref. \cite{sd10}%
. We show that the framework can be extended consistently on a general $(p,q)
$ type Cantor set which is derived recursively from the unit interval\textbf{%
\ }$I$\textbf{\ }first dividing it into $r$ number of equal closed intervals
and then deleting $q$ number of open intervals so that $p+q=r$. We show that
the non-archimedean valuation is related to the Cantor function $\phi (x):$ $%
I\rightarrow $ $I$\textbf{\ }such that $\phi ^{\prime }(x)=0$ a.e. on $I$
with discontinuities at the points $x\in C$. In the non-archimedean
framework $\phi (x)$ is shown to be extended to a locally constant function for {\em any} $x\in I$. Using the non-archimedean valuation we also construct a valued measure on $C$
which is shown to give the finite Hausdorff measure of the set. The
variability of the locally constant $\phi (x)$ is reinterpreted in the usual
topology as an effect of relative infinitesimals which become dominant by
inversion at an appropriate scale.

The paper is organised as follows. In Sec.$2,$ we give a brief sketch of the
details of a $(p,q)$ Cantor set and the corresponding Cantor function. 
In Sec.$3$, we give an
outline of the scale invariant analysis and the valued measure on $C$. The
concept of relative infinitesimals and new absolute values are introduced in
Sec.$3.1.\ $The valued measure is constructed in Sec.$3.2.\ $In Sec.$3.3,$
we define scale invariant differentiability on $C.\ $ In
Sec.$4$, we discuss the example of $(p,q)$ Cantor set and show how the
valuation is identified with the Cantor function. We also show explicitly
how the variability of a non-archimedean locally constant function $\phi
(x)\ $is exposed in the usual topology when the ordinary differential
equations in $I\ $get extended to scale invariant equations in appropriate
logarithmic (infinitesimal) variables.\bigskip

\section{\hspace{-2ex}{$(p,q)$\textbf{\ Cantor Set and Cantor Function }}%
\protect\bigskip}

\indent\indent To make the article self-contained we present here a brief
review of the Cantor set and Cantor function. We note that the middle third
Cantor set and the related Cantor function are well discussed in the
literature. The definition of $(p,q)$ Cantor set $C$ and the corresponding
Cantor function are analogous to the above case except for minor
modifications.

We divide the unit interval $I$ = [0,1] into r number of closed subintervals
each of length $\frac{1}{r}$ and delete q number of open subintervals from
them so that $p+q=r$.  The deletion of $q$ open intervals may be accomplished by 
an application of an iterated function system (IFS) of  similitudes of the form $f={f_i: I\rightarrow I, \ i=1, 2, \ldots p}$, where $f_i(x)={\frac{1}{r}}(x+\alpha_i)$ and $\alpha_i$ assumes values from the set $\{0,1,2, \ldots,(r-1)\}, \ i=1,2,\ldots, p$. We note incidentally that there are, in fact, $^rC_q$ distinct IFS each of which has $C$ as the unique limit set, viz., $C=f(C)$.
   
To construct the limit set $C$ explicitly, we note that  the set $I$, after the first iteration,  is reduced to $%
\textrm{I}=\underset{n=1}{\overset{p}{\cup }}F_{1n}$ consisting of p
number of closed intervals $F_{1n}$,  so that the length of the deleted intervals is $\frac{q}{r}.$ Iterating the above steps in each of the closed intervals $%
\textrm{F}_{1n}$ ad infinitum we get the desired Cantor set $C$ = $\overset{%
\infty }{\underset{n=0}{\cap }}\underset{m=0}{\overset{p^{n}}{\cup }}F_{nm},$
$F_{00}=$ $I$\textbf{. }The length of the deleted intervals at the $n$ th
iteration is $\frac{q}{p}\ [\frac{p}{r}+(\frac{p}{r})^{2}+\cdots +(\frac{p}{r%
})^{n}]=\frac{q[1-(\frac{p}{r})^{n}]}{r(1-\frac{p}{r})}=[1-(\frac{p}{r}%
)^{n}]\rightarrow 1$ as $n\rightarrow \infty .\ $Thus the Lebesgue measure
of $C$ is zero. However the Hausdorff s-measure of $C$, given by \
\begin{equation}
\mu _{s}[\ C\ ]=\underset{\delta \rightarrow 0}{\lim }\inf \underset{i}{%
\Sigma }[\ d(U_{i})\ ]^{s}
\end{equation}%
where $d(U_{i})$ is the diameter of the set $U$ and infimum is taken over
all countable $\delta -$ covers\textbf{\ }$I_{\delta }=\{U_{i}\}$ such that $%
C$ $\subset \cup $ $U_{i}$ is finite for the unique value of $s$ satisfying
the scaling equation $p=$ $r^{s}.$ The Hausdorff dimension of C thus equals $%
\frac{\log p}{\log r}$.

Next we define the Cantor function $\phi :[0,1]\rightarrow \lbrack 0,1].\ $
Let $\phi (0)=0,\ \phi (1)=1.$ Assign $\phi (x)$ a constant value on each of
the deleted open intervals (including the end points of the deleted
interval). The constant values are assigned in the following manner.

At the first iteration we set $\phi (x)=\frac{t}{p},\ t=1,2,\ldots ,q.\ $At
the second step there are $q(1+p)$ deleted intervals and so we set $\phi (x)=%
\frac{t}{p^{2}},\ t=1,2,\ldots ,q(1+p)$ at each of the deleted intervals
respectively. The number of deleted intervals at the $n$ th step is $%
q(1+p+p^{2}+\cdots +p^{n-1})=\frac{q(1-p^{n})}{1-p}=N$ (say) so that the
value assigned to $\phi (x)$ at each deleted intervals $($including the end
points$)$ are $\phi (x)=\frac{t}{p^{n}},\ t=1,2,.....N.\ $Next, let $x$ $\in
C$. Then for each $k,x$ belongs to the interior of exactly one of the $p^{n}$
remaining closed intervals each of length $\frac{1}{r^{k}}.\ $Let $[\alpha
_{k,}\beta _{k}]$ be one such intervals. Then
\begin{equation}
\beta _{k}-\alpha _{k}=\frac{1}{r^{k}}
\end{equation}%
Further, $\phi $ is already defined at the $2N$ end points of the left over
intervals so that%
\begin{equation}
\phi (\beta _{k+j})=\phi (\alpha _{k})+\frac{1}{p^{k}}
\end{equation}%
where $0\leq j\leq p-q$ and $\alpha _{1}=0.\ $At the next iteration,
assuming $x\in \lbrack \alpha _{k+1},\beta _{k+j+1}],$ $\alpha _{k}=\alpha
_{k+1},\ $say, we have $\phi (\alpha _{k+j})\leq \phi (\alpha _{k+1})<\phi
(\beta _{k+j+1})\leq \phi (\beta _{k+j}).\ $Define $\phi (x)=\underset{%
k\rightarrow \infty }{\lim }\phi (\alpha _{k})=\underset{k\rightarrow \infty
}{\lim }\phi (\beta _{k+j+1}).\ $Then $\phi :[0,1]\rightarrow \lbrack 0,1]$
is a continuous, non-decreasing function. Also $\phi ^{\prime }(x)=0$ for $%
x\in $ $I\backslash C$ when it is not differentiable at any \ $x\in C.$

\bigskip

\section{\hspace{-2ex}{\textbf{Non-archimedean analysis}}}

\subsection{\hspace{-3ex}{\ Absolute Value}\protect\bigskip}

\begin{definition}
Let $x\in C\subset $ $I.$ For an arbitrary small $x\rightarrow 0^{+},\exists
$ an $\varepsilon \in I$ and a $1$-parameter family of $\tilde{x}$ in $%
I\backslash C$ such that $0<\tilde{x}<\varepsilon <x$ and%
\begin{equation}
\frac{\tilde{x}}{\varepsilon }=\lambda (\varepsilon )\ \frac{\varepsilon }{x}
\label{eq4}
\end{equation}%
where the real constant $\lambda \ (0<\lambda \leq 1)$ may depend on $%
\epsilon $. The set of such $\tilde{x}$'s satisfying the inversion law $($%
\ref{eq4}$)$ is called the set of relative infinitesimals \cite{sd10,aro12} in%
\textbf{\ }$I$\textbf{\ }relative to the scale $\varepsilon $ and is denoted
as $\mathbf{I}_{0}^{+}=\{\tilde{x}\mid 0<\tilde{x}<\varepsilon <x,\ \tilde{x}%
=\lambda (\varepsilon )\ \frac{\varepsilon ^{2}}{x}\}.\ $Two relative
infinitesimals $\tilde{x}$ and $\tilde{y}$ $\ $must satisfy the condition $0<%
\tilde{x}<\tilde{y}<\tilde{x}+\tilde{y}<\varepsilon .$
\end{definition}

The non-empty set $\mathbf{I}^{+}=\{\ \frac{\tilde{x}}{\varepsilon },\ \varepsilon
\rightarrow 0\ \}$ is called the set of scale free infinitesimals.

\begin{definition}
Because of the disconnectedness of $C$, to each $x\in $ $C$,\textbf{\ }$%
\exists $ $\mathbf{I}_{\varepsilon }(x)=(x-\varepsilon ,x+\varepsilon
)\subset $ $I$\textbf{, }$\varepsilon >0$ such that $C\cap \mathbf{I}%
_{\varepsilon }(x)=\{x\}.\ $Points in $\mathbf{I}_{\varepsilon }(x)$ are
called the relative infinitesimal neighbours in $I$\textbf{\ }of $x\in $
C.
\end{definition}

\begin{lemma}
$\mathbf{I}_{\varepsilon }(x)=x+$\textbf{\ }$\mathbf{I}_{0},$\textbf{\ }$%
\mathbf{I}_{0}=\mathbf{I}_{0}^{+}\cup \ \mathbf{I}_{0}^{-},\ \mathbf{I}%
_{0}^{-}=\{\ -\tilde{x}\mid \tilde{x}\in $\textbf{\ }$\mathbf{I}_{o}^{+}\
\}. $ Further $\exists $ a bijection between $\mathbf{I}_{0}^{+}$ and $(0,1)$
for a given $\varepsilon .$
\end{lemma}

\begin{proof}
Let $y$ $\in $ $\mathbf{I}_{\varepsilon }(x)$. Then $y=x\pm \tilde{x}$, $%
0<\tilde{x}<\varepsilon <z,$ so that $\tilde{x}$ = $\lambda \
\frac{\epsilon ^{2}}{z}$ for a fixed $z$ and a variable $\lambda .\ $Thus $y$
$\in x+\mathbf{I}_{0}.\ $The other inclusion also follows similarly.
Finally, the bijection is given by the mapping $\tilde{x}\rightarrow \frac{%
\tilde{x}}{\varepsilon }.$
\end{proof}

\begin{definition}
Given $\tilde x$ $\in $\textbf{\ }$\mathbf{I}_{o},$ we define a scale free
absolute value of $\tilde x$ by $v:\mathbf{I}_{0}\rightarrow \lbrack 0,1]$ where%
\begin{equation}
v(\tilde x)=\left\{
\begin{array}{c}
\log _{\varepsilon ^{-1}}\frac{\varepsilon }{\mid \tilde{x}\mid },\ \ \ \ \
\tilde x\neq 0 \\
0,\ \ \ \ \ \ \ \ \ \ \ \ \ \ \tilde x=0%
\end{array}%
\right.
\end{equation}%
\ \ as $\varepsilon \rightarrow 0^{+}.$\
\end{definition}

\begin{lemma}
$v$ is a non-archimedean semi-norm over $\mathbf{I}_{0}.$
\end{lemma}

\begin{notation}
By semi-norm we mean $\left( i\right) \ v(\tilde x)>0,\ \tilde x$ $\neq 0$.$\ \ (ii)\
v(-\tilde x)=v(\tilde x)$. $(iii)\ v(\tilde x+\tilde y)$ $\leq \max \{\ v(\tilde x),v(\tilde y)\ \}.$ Property $(iii)$
is called the strong $($ultrametric$)$ triangle inequality \cite{na9}.
\end{notation}

\begin{proof}
The case $(i)$ and $(ii)$ follow from the definition. To prove $(iii)$ let $%
0<\tilde x\leq \tilde y<\tilde x+\tilde y<\varepsilon .\ $Then $v(\tilde y)$ $\leq v(\tilde x)$ and hence $v(\tilde x+\tilde y)=\log
_{\varepsilon ^{-1}}\frac{\varepsilon }{\tilde x+\tilde y}\leq \log _{\varepsilon ^{-1}}%
\frac{\varepsilon }{\tilde x}=v(\tilde x)=\max \{\ v(\tilde x),v(\tilde y)\ \}.$ Moreover, $v(\tilde x-\tilde y)=v(\tilde x+(-\tilde y))%
\leq \max \{\ v(\tilde x), v(\tilde y)\ \}.$
\end{proof}

\begin{example}
Let $\varepsilon =e^{-n},x=k\varepsilon =\varepsilon ^{-t}.\varepsilon $
where $t\rightarrow 0^{+}$ for an $k\approx 1.$ Consider a subset of the open interval $%
I_{\varepsilon }=(0,\varepsilon )$ consisting of $q$ open subintervals $%
I_{j},j=1,2,\ldots ,q$ each of length $\frac{\varepsilon }{r}, \ r>q$. Let $\tilde{I}%
_{j}\subset (0,1)$ be the image of $I_{j}$ under rescaling $\tilde{x}%
\rightarrow \frac{\tilde{x}}{\varepsilon }$. The relative infinitesimals $%
\tilde{x}_{j}\in I_{j}$ are given by $\tilde{x}_{j}=\lambda
_{j}.k^{-1}:=\varepsilon ^{\mu _{j\text{ }}t}$ where $\lambda _{j}\in \tilde{I%
}_{j}$ and $\mu_j=1+t^{-1}\frac{\log \lambda _j}{\log \varepsilon}$. Then $v(\tilde{x}_{j})=\mu _{j\text{ }}t.$ 
\end{example}

\begin{definition}
The set $B_{r}(a)=\{\ x\mid v(x-a)<r\ \}$ is called an open ball in $\mathbf{%
I}_{0}.\ $The set $\bar{B}_{r}(a)=\{\ x\mid v(x-a)\leq r\ \}$ is a closed
ball in $\mathbf{I}_{0}.$
\end{definition}

\begin{lemma}
$(i)$ Every open ball is closed and vice-versa (clopen ball) $(ii)$ every
point $b\in B_{r}(a)$ is a centre of $B_{r}(a).\ (iii)$ Any two balls in $%
\mathbf{I}_{0}$ are either disjoint or one is contained in another. $(iv)$ $%
\mathbf{I}_{0}$ is the union of at most of a countable family of clopen
balls.
\end{lemma}

Proof follows directly from the ultrametric inequality and the fact that $%
\mathbf{I}_{0}$ is an open set. It also follows that in the topology
determined by the semi-norm, $\mathbf{I}_{0}$ is a totally disconnected set.
It is also proved in \cite{sd10} that a closed ball in $\mathbf{I}_{0}$ is
compact. As a result, $\mathbf{I}_{0}$ is the union of countable family of
disjoint closed (clopen) balls, in each of which $v\left( \tilde x\right) $ can
have a constant value. With this assumption, $v$ : $\mathbf{I}%
_{0}\rightarrow \lbrack 0,1]$ is discretely valued. Next, to restore the
product rule viz : $v(\tilde x\tilde y)=v(\tilde x).v(\tilde y),\ $we note that given $\tilde x$ and $%
\varepsilon ,\ 0<\tilde x<\varepsilon ,$ there exist $0<\sigma (\varepsilon )<1$
and $a:$ $\mathbf{I}_{0}^{+}\rightarrow R$ such that
\begin{equation}
\frac{\tilde x}{\varepsilon }=\varepsilon ^{\sigma ^{a(\tilde x)}}.\varepsilon
^{t(\tilde x,\varepsilon )}
\end{equation}%
so that $v(\tilde x)=\sigma ^{a(\tilde x)}$ for an indeterminate vanishingly small $t:$ $%
\mathbf{I}_{0}\rightarrow R$ i.e. $t(\tilde x,\varepsilon )\rightarrow 0$ as $%
\varepsilon \rightarrow 0^{+}.$ For the given Cantor set $C$ there is a
unique $($natural$)$ choice of $\sigma $ dictated by the scale factors of $C$
viz : $\sigma =p^{-n}=r^{-ns},\ s=\frac{\log p}{\log r}$, for some natural number $n$.

The mapping $a(\tilde x)$ is a valuation and satisfies $(i)$ $a(\tilde x\tilde y)=a(\tilde x)+a(\tilde y),$ $%
(ii)$ $a(\tilde x+\tilde y)\geq \min \{\ a(\tilde x),a(\tilde y)\ \}$. Now discreteness of $v(\tilde x)$
implies range $\{\ a(\tilde x)\ \}=\{\ a_{n}\mid n\in Z^{+}\ \}.\ $Again for a
given scale $\varepsilon $, $\mathbf{I}_{0}^{+}$ is identified with a copy
of $(0,1)$ (by Lemma 1) which is clopen in the semi-norm. Thus $\mathbf{I}%
_{0}^{+}$ is covered by a finite number of disjoint clopen balls $B(\tilde x_{n})$ $%
($say$)$, $\tilde x_{n}\in $ $\mathbf{I}_{0}^{+}.\ $Because of finiteness, values
of $a(\tilde x)$ on each of the balls can be ordered $0=a_{0}<a_{1}<\cdots
<a_{n}=s_{0}\ $(say)$.\ $Let $v_{0}=v(B(\tilde x_{n}))=\sigma ^{s_{0}}.$ Then
we can write $v_{i}=v(B(\tilde x_{i}))=\alpha _{i}v_{0}=\alpha _{i}\ \sigma^{s_{0}}$ for an ascending sequence $\alpha _{i}>0,\ i= 0,1,\ldots ,n.$ We
also note that $a_{0}=0$ corresponding to the unit $\tilde x_{u}$ so that $v\left(
\tilde x_{u}\right) =1$.

From equation $(6)$ we have $\frac{\tilde x_{u}}{\varepsilon }=\varepsilon
^{1+t(\tilde x,\varepsilon )}$ and so it follows that $\tilde x\in $ $\mathbf{I}_{0}^{+}$
will admit a factorization
\begin{equation}
\frac{\tilde x}{\varepsilon }=\frac{\tilde x_{i}}{\varepsilon }.\ \frac{\tilde x_{u}}{\varepsilon
^{2}}
\end{equation}%
since $\tilde x\in B(\tilde x_{i})$ for some $i$.\newline
Thus
\begin{equation}
\tilde x=\tilde x_{i}\ (1+\tilde x_{\varepsilon })
\end{equation}%
where $\tilde x_{u}=\varepsilon ^{2}(1+\tilde x_{\varepsilon }),$ \ $\tilde x_{\varepsilon }\in $
$\mathbf{I}_{0},$ so that $v(\tilde x)=v(\tilde x_{i}),$ as $v(\tilde x_{\varepsilon })<1.$

We thus have,

\begin{theorem}
$v$ is a discretely valued non-archimedean absolute value on $\mathbf{I}%
_{0}^{+}.$ Any infinitesimal $\tilde x$ $\in $ $\mathbf{I}_{0}^{+}$ have the
decomposition given by equation $(8)$ so that $v\ $has the canonical form
\begin{equation}
\ v(\tilde x)=\alpha _{i}\ \sigma ^{s_{0}}\mathbf{\ ,\ }\tilde x\in \mathbf{B}(x_{i})
\end{equation}
\end{theorem}

\begin{definition}
The infinitesimals given by equation $(8)$ and having absolute value $(9)$ are called valued infinitesimals.
\end{definition}

We now make use these valued infinitesimals to define a non-trivial absolute
value on C in the following steps.

$(i)$ Given $x\in C$ define a set of multiplicative neighbours of $x$ which
are induced by the valued infinitesimals $\tau \in \mathbf{I}_{0}^{+}$ by
\begin{equation}
X^{\tau }_{\pm }=x.\ x^{\mp v(\tau )}
\end{equation}%
where $v(\tau )=\alpha _{n}\sigma ^{s_{0}}$ and $\alpha _{n}=\alpha _{n}(x)$
may now depend on $x$. We note that the non-archimedean topology induced by $%
v$ makes the infinitesimal neighbourhood of $0^{+}$ in $I$ totally
disconnected. Equation $(10)$ thus introduces a finer infinitesimal
subdivisions in the neighbourhood of $x$ $\in C.$

$(ii)$ We define the new absolute value of $x$ $\in C$ by $\qquad \qquad $ \
\ \
\begin{equation}
\ \ \parallel x\parallel =\inf \log _{x^{-1}}\frac{X_{+}}{x}=\inf \log
_{x^{-1}}\frac{x}{X_{-}}
\end{equation}%
so that $\parallel x\parallel =\sigma ^{s}$ where $\sigma ^{s}=\inf \alpha
_{n}\sigma ^{s_{0}}$ and the infimum is over all $n$. It thus follows
that

\begin{corollary}
$\parallel .\parallel \ :C\rightarrow R_{+}$ is a non-archimedean absolute
value.
\end{corollary}

\bigskip

\subsection{\hspace{-2ex}{\textbf{Valued measure}}\protect\bigskip}

\indent\indent We define the valued measure $\mu _{v}:C\rightarrow R_{+}$ by

$(a)$ $\mu _{v}(\phi )=0,\ \ \phi $ the null set.

$(b)$ $\mu _{v}[(0,x)]=\parallel x\parallel \ \ $ when $x\in C.$

$(c)$ For any $E\subset C,$ \ we have
\begin{equation}
\mu _{v}(E)=\underset{\delta \rightarrow 0}{\lim }\inf \Sigma \{\
d_{na}(I_{i})\ \}
\end{equation}%
where ${I}_{i}\in $\textbf{\ }$\mathbf{\tilde{I}}_{\delta }$ and the
infimum is over all countable $\delta -$ covers $\mathbf{\tilde{I}}_{\delta }
$ of $E$ by clopen balls and$\ d_{na}({I}_{i})=$ the non-archimedean
\textquotedblleft diameter\textquotedblright\ of ${I}_{i}$\ $=\sup
\{\parallel x-y\parallel :x,y\in $ $\mathbf{I}_{i}\}.$\newline
It follows that $\mu _{v}$ is a metric (Lebesgue)\textbf{\ }outer measure on
$C$ realized as a non-archimedean space. \newline
Now, denoting the diameter in the usual sense by $d(\mathbf{I}_{i})$, one
notes that $d_{na}(I)\leq \{\ d(\mathbf{I}_{i})\ \}^{s},$ since $x,y$ $\in C$
and $\mid x-y\mid =d$ imply $\parallel x-y\parallel =\varepsilon ^{s}\leq
d^{s}$ for a suitable scale $\varepsilon \leq d\leq \delta $. Using this
inequality one can show that \cite{sd10}
\begin{equation}
\mu _{v}(E)=\mu _{s}(E)
\end{equation}%
for any subset $E$ $\subset C.$ Finally, for $s=\dim [C],$ $\mu _{s}(C)=1$
and so $\mu _{v}(C)=1.$ Thus the valued measure selects naturally
the dimension of the Cantor set.

\bigskip

\subsection{\hspace{-2ex}{\textbf{Differentiability}}}

\bigskip

\indent\indent To discuss the formalism of the Calculus on $C$ we change the
notations of section $3.1$ a little. Let $X$ denote a valued infinitesimal
while an arbitrarily small real $x$ $\in $\textbf{\ }$I$\textbf{\ }denote
the scale\textbf{\ }$\varepsilon $. The set of infinitesimals is covered by $%
n$ clopen balls B$_{n}$ in each of which $v$ is constant. Let
\begin{equation}
\tilde{v}_{n}(x)=v(\ X_{n}(x)\ )=\log _{x^{-1}}{\frac{x}{X_{n}}}=\alpha
_{n}\ x^{s_{0}}
\end{equation}%
so that $X_{n}=x.\ x^{\tilde{v}_{n}(x)}\in B_{n}.$ For each $x,$ $\tilde{v}%
_{n}$ is constant on $B_{n}.$

\begin{definition}
A function $f:C\rightarrow $\textbf{\ }$I$\textbf{\ }is said to be
differentiable at $x_{0}\in C$ if $\exists $ a finite $l$ such that $%
0<\parallel x-x_{0}\parallel <\delta \Rightarrow $%
\begin{equation}
\left\vert \frac{\mid f(x)-f(x_{0})\mid }{\parallel x-x_{0}\parallel }%
-l\right\vert <\varepsilon
\end{equation}%
for $\varepsilon >0$ and $\delta (\varepsilon )>0$ and we write $%
f^{\prime }(x_{0})=l.\ $
\end{definition}

Now $\parallel x-x_{0}\parallel =\inf \tilde{v}_{n}(x-x_{0})=\log
_{x_{0}^{-1}}{\frac{x_{0}}{X}},$ where the valued infinitesimal $X$ $\in $ $%
\tilde{B}$, an open sub-interval of $[0,1]$ in the usual topology and $%
\tilde{B}$ is the ball which corresponds to the infimum of $\tilde{v}_{n}$.
Further $f(x)-f(x_{0})=(\log x_{0})^{-1}\tilde{f}(X),$ since $%
x=x_{0}.x_{0}^{\pm v(x)},$ and $\tilde{f}$ is a differentiable function on $\tilde{B}$ in
the usual sense. Thus equation $(15)$, viz., the equality $ f^{\prime }(x_{0})=l,\ $ extends over $\tilde{B}$ as a scale
free differential equation
\begin{equation}
\frac{d\tilde{f}}{d\log X}=l
\end{equation}

\begin{definition}
: Let $f:C\rightarrow C$ be a mapping on a Cantor set $C$ to itself. Then $f$
is differentiable at $x_{0}\in C$ if $\exists $ $l$ such that given $%
\varepsilon >0,\exists $ $\delta >0$ so that%
\begin{equation}
\left\vert \frac{\parallel f(x)-f(x_{0})\parallel }{\parallel
x-x_{0}\parallel }-l\right\vert <\varepsilon
\end{equation}%
when $0<\parallel x-x_{0}\parallel <\delta .$
\end{definition}

As before we write $f^{\prime }(x_{0})=l$ (with an abuse of notation). It follows that 
the above equality now extends to a scale free equation of the form
\begin{equation}
\frac{d\log \tilde{f}(X)}{d\log X}=l
\end{equation}%
where notations are analogus to above.

\begin{remark}
The discrete point like structures of $C$ are replaced by infinitesimal open
intervals over which the ordinary continuum calculus is carried over on
logarithmic variables via the scale invariant non-archimedean metric.
\end{remark}

In the next section we show that the Cantor function is a locally constant
continuously differentiable function in the new sense. We also give its
reinterpretation in the usual topology.

\bigskip

\section{\hspace{-2ex}\textbf{Cantor function revisited}}

\bigskip

\indent\indent We first show that the value $v(x)$ awarded to the valued
infinitesimals $X$ $\in $ $B_{i}$, $i=1,2,\ldots ,n$ is given by the Cantor
function $\phi :$ $I$\textbf{\ }$\rightarrow $ $I$\textbf{\ }with\textbf{\ }%
points of discontinuity in $\phi ^{\prime }(x),$ in the usual sense, are in $%
C$. In the new formalism this discontinuity is removed in a scale invariant
way using logarithmic differentiability over (valued) infinitesimal open
line segments replacing each $x\in C$. Our definition of $v(x)$ is guided by
the given Cantor set $C$ so as to retrieve the finite Hausdorff measure uniquely
via the construction of the valued measure.

Let us denote the valued scale free infinitesimals by $[0,1)$, denoted here
by $\tilde{C}.\ $The interval $[0,1)\ $here is a copy of the scale free
infinitesimals ${\bf I}^{+}$\ for an arbitary small $\varepsilon _{0}$ (say).
The valued infinitesimals in $[0,1)\ $ then introduce  a new set of
scales of the form $r^{-n}$ (in the unit of $\varepsilon _{0}$) so that the
scales introduced in definition $1$ are now parameterized as $\varepsilon
=\varepsilon _{0}\ r^{-n}.\ $The choice of the `secondary' scales $r^{-n}\ $%
are motivated by the finite level Cantor set $C.\ $At the ordinary level
i.e. at the scale $1\ $(corresponding to $n=0$), there is no valued
infinitesimal (at the level of ordinary real calculus) except the trivial $0
$. So relative to the finite scale (given by $\delta =\frac{\varepsilon }{%
\varepsilon _{0}}=1$) $\ [0,1)$ reduces to the singleton $\{0\}$.  At the
next level, we choose the smaller scale $\delta =\frac{1}{r}.$ Consequently,
elements in $[0$,$\frac{1}{r})$ are undetectable and identified with $0$,
again in the usual sense. Presently we have, however, the following.

We assume that the void (emptiness) of $0$ reflects in an inverted manner
the structure of the Cantor set $C$ that is available at the finite scale.
That is to say, at the first iteration of $C$ from $I$\textbf{, }$q$ open
intervals are removed leaving out $p$ closed intervals F$_{1n},\
n=1,2,\ldots ,p.$ At the scale $\frac{1}{r}$ in the void of $\tilde{C},$ on
the other hand, there now emerges (by \textquotedblleft
inversion\textquotedblright ) $q$ open islands (intervals) $\mathbf{I}%
_{1i},\ i=1,2,\ldots ,q.\ $By definition, $\mathbf{I}_{1i}$ contains, for
each $i$, the so called valued infinitesimals $X_{i}$ which are assigned the
values $v(X_{i})=\phi (X_{i})=\frac{i}{p},\ i=1,2,\ldots ,q,X_{i}\in I_{1i}.$

We note that at the scale $\delta =\frac{1}{r}$, there are $p$ voids in $%
\tilde{C}$. At the next level of the scale $\frac{1}{r^{2}},$ there emerges
again in each void $q$ islands of open intervals, so that there are now $pq\
$number of total islands $\mathbf{I}_{2i},\ i=1,2,\ldots ,pq.\ $The value
assigned to each of these valued islands of infinitesimals are $%
v(X_{j})=\phi (X_{j})=\frac{j}{p^{2}},\ j=1,2,\ldots ,pq,\ $where $X_{j}\in $
$\mathbf{I}_{2j}.\ $Continuing this iteration, at the $n$ th level, the
 (secondary) scale is $\delta =\frac{1}{r^{n}}$ and the number of open intervals $\mathbf{%
I}_{nj}$ of infinitesimals are now $q(1+p+p^{2}+\cdots +p^{n})=N$ $($say$)$
with corresponding values%
\begin{equation}
v(X_{j})=\phi (X_{j})=\frac{j}{p^{n}},\ j=1,2,\ldots ,N\
\end{equation}%
where $X_{j}\in $ $\mathbf{I}_{nj}.\ $Thus $v$ and hence the Cantor function
$\phi $ is defined on the \textquotedblleft inverted Cantor
set\textquotedblright\ $\tilde{C}$ = $\underset{n}{\cap }\ \underset{j}{\cup
}$ $\mathbf{I}_{nj}$ and is extended to $\phi :$ $I$ $\rightarrow $\textbf{\
}$I$\textbf{\ \ }by\textbf{\ }continuity following equations like equation $%
(3)$. We note that the absolute value $\Vert .\parallel $ awarded to each
block of the Cantor intervals $F_{nk}$ are%
\begin{equation}
\parallel F_{nk}\parallel =r^{-ns}
\end{equation}%
for each $k=1,2,\ldots ,p^{n}$ where $C=$ $\underset{n}{\cap }\ \underset{k}{%
\cup }$ ${F}_{nk}$ and so $s=\frac{\log p}{\log r},$ since the valued
set of infinitesimals induces fine structures to an element in $F_{nk}$ viz.
for a $y$ $\in $ $F_{nk},$ we now have the infinitesimal neighbours $Y_{\pm
}^{j}=y.\ y^{\mp jp^{-n}},j=1,2,\ldots ,N.$

Clearly, the absolute value in equation $(20)$ corresponds to the minimum of
$v(x)$ at the $n$th iteration. Thus the valuations defined as the
associated Cantor function leads to a valued measure on $C$ that equals the
corresponding Hausdorff measure with $s=\frac{\log p}{\log r}.$

Let us now recall that the solutions of $\phi ^{\prime }(x)=0$ in a
non-archimedean space are locally constant functions \cite{na9}. To show that Cantor
function $\phi :$ $I$ $\rightarrow $\textbf{\ }$I$ is a locally constant
function\textbf{, }let us recall that the Cantor set $C$ is constructed
recursively as$\ C=$ $\underset{n}{\cap} \ \underset{k}{\cup }F_{nk}.$ The set $I\ ,\ $on
the other hand, is written as
$I=$ $\underset{n}{\cap }\ [\ (\overset{p^{n}}{\underset{k=1}{\cup }}$ $\tilde{F%
}_{nk})\ \cup \ ($ $\overset{N}{\underset{j=1}{\cup }}\ $ $\mathbf{I}_{nj})\ ],$
the open interval $\tilde{F}_{nk}$ being $F_{nk}$ with end points removed
(recall that $\mathbf{I}_{nj}$ are closed in the ultrametric topology). By definition $v(\mathbf{I}%
_{nj})=a_{nj}$ a constant for each $n$ and $j$. We set $v($ $\tilde{F}%
_{nk})=0$ as $n$ $\rightarrow \infty .$ This equality is to be understood in
the following sense. At an infinitesimal scale $\varepsilon _{0}\rightarrow
0^{+}$ the zero value of $\tilde{F}_{nk}$ becomes finitely valued
recursively for each $n\ $since a Cantor point $x\in $ $C$ is replaced by a
copy of the (inverted) Cantor set $\tilde{C}$ with finite number of closed
intervals like $\mathbf{I}_{nj}$. The derivatives of $\phi $ vanishes not only for
each $n$ and $j$ but {\em even as} $n$ $\rightarrow \infty $ (and $\varepsilon
\rightarrow 0,\ $for each arbitrarily small but fixed $\varepsilon _{0})$. 
Thus, the equality $\phi ^{\prime }(x)=0$ on  $I/C$, in the ordinary sense, 
gets extended  to every $x\in C$ when the Cantor set is reinterpreted as a 
 nonarchimedean space. The removal of the usual derivative discontinuities is 
 also explained dynamically as due to the fact that the approach to an actual Cantor 
 set point $x$ is accomplished in the nonarchimedean setting by inversion. That is to say, as a variable $X\in I$ approaches $x\in C$, the usual linear shift in $I$ is replaced by  
infinitesimal hoppings between two nieghbouring elements of the form $X_+/x\propto x/X_-$.

The variability of the locally constant function $\phi: I\rightarrow I $ may, however, be
captured in the usual topology as follows. \newline
Indeed, we show that
\begin{equation}\label{ord}
\frac{d\phi }{dx}=0
\end{equation}%
for finite values of $x\in I$ is transformed into
\begin{equation}\label{nonord}
\frac{d\phi }{dv(\tilde{x})}=-O(1)\phi
\end{equation}%
for an infinitesimal $\tilde{x}$ satisfying $\frac{x}{\varepsilon }=\lambda \
\frac{\varepsilon }{\tilde{x}}=\varepsilon ^{-v(\tilde{x})},\ 0<\tilde{x}%
<\varepsilon \leq x,\ x\rightarrow 0^{+},\ x\in I, \ \lambda >0,\ $  when one interprets $0$ in relation to the scale $\varepsilon $ as $O(\ \delta= \frac{\varepsilon ^{2}}{x%
}\log \varepsilon ^{-1}\ )$.   However, this follows once one notes that eq(\ref{ord}) means, in the ordinary sense, $d\phi=0=O(\delta), \ dx\neq 0$, for a finite $x\in I$. But, as $x\rightarrow \varepsilon$, that is, as $dx\rightarrow 0=O(\delta)$, the ordinary variable $x$ is replaced by the ultrametric extension $x=\varepsilon.\varepsilon^{-v(\tilde x)}$ so that $d\log x=dv(\tilde x)\log\varepsilon^{-1}=O(\delta)$. On the other hand, the constant function $\phi$ (eq(\ref{ord})), now, in the presence of smaller scale infinitesimals, has the form $\phi=\phi_0\varepsilon ^{k_0v(\tilde x)}$ for a real constant $k_0$. Eq(\ref{nonord}) thus follows. The
variability of $\phi (x)$ in the usual topology is thus explained as an
effect of the relative infinitesimals which are {\em insignificant} relative to
the finite scale of  $x\in C,$ but attain a dominant status in the
appropriate logarithmic variable $v(\tilde{x})=\log _{\varepsilon ^{-1}}\frac{%
\varepsilon }{\tilde{x}}.$ It is also of interest to compare the present
case with computation. In the ordinary framework, the scale $\varepsilon $
stands for the level of accuracy in a computational problem. The
infinitesimals in $(0,\varepsilon )$ are \textquotedblleft
valueless\textquotedblright\ in the sense that these have no effect on the
actual computation. The open interval  $(0,\varepsilon )$ is thus effectively 
indentified with $\{0\}$. In the present framework, the zero element $0$ is,
however, identified with a smaller interval of the form $(0,\delta )$ where $%
\delta =\eta \varepsilon \log \varepsilon ^{-1}$ and $0<\eta \lesssim 1.$
The valued infinitesimals in the interval $(\delta ,\varepsilon )$ are already shown
to have significant influence on the structure of the Cantor set. The
variability of $\phi (x)$ as given by equation $(22)$ is revealed, on the other hand,  in relation to an infinitesimal variable lying in $(0,\delta ).$

Finally, we verify the emergence of equation (22) from the classical Cantor
function equation $(2)$ and $(3)$ viz. : (we choose $j=0$ for simplicity)

\begin{equation}
\phi (\beta _{k})-\phi (\alpha _{k})=\frac{1}{p^{k}}\text{ and }\beta
_{k}-\alpha _{k}=\frac{1}{r^{k}}
\end{equation}%
We have
\begin{equation}
\phi (\beta _{k})-\phi (\alpha _{k})=\left( \frac{r}{p}\right) ^{k}(\text{ }%
\beta _{k}-\alpha _{k})
\end{equation}%
Let $\phi (\beta _{k})=\tilde{\phi}_{+},$ $\ \phi (\alpha _{k})=\tilde{\phi}%
_{-},$ $\ \beta _{k}=x_{+,}$ $\ \alpha _{k}=x_{-}.$ Suppose also that $%
r^{k}(x_{+}-x)\rightarrow $ $k\log \sigma _{+},$ $r^{k}(x-x_{-})\rightarrow $
$k\log \sigma _{-}$ , $p^{k}(\tilde{\phi}-\tilde{\phi}_{-})\rightarrow k\log
\phi _{-}^{\prime }$ \ and $p^{k}(\tilde{\phi}_{+}-\tilde{\phi})\rightarrow
k\log \phi _{+}^{\prime }$  as $k\rightarrow \infty .$

Equation $(24)$ becomes%
\begin{equation}
\log \phi _{+}^{\prime }+\log \phi _{-}^{\prime }=\log \sigma _{+}+\log
\sigma _{-}
\end{equation}%
which leads to
\begin{equation}
\frac{\log \phi _{+}^{\prime }}{\log \sigma _{+}}=\frac{\log \phi
_{-}^{\prime }}{\log \sigma _{-}}=\frac{\log \phi _{+}^{\prime }+\log \phi
_{-}^{\prime }}{\log \sigma _{+}+\log \sigma _{-}}=1
\end{equation}%
Equation $(26)$ is essentially the left and right brunches of equation $(21)$ at $x$ $%
\in $ $C$, in appropriate logarithmic variables,  where the multiplicative neighbours of $x$, in the present derivation, is given by the limiting form of the Cantor function defined by%
\begin{equation}
\phi _{+}^{\prime }=\sigma ^{1+i},\ \phi _{-}^{\prime }=\sigma ^{-(1+i)},\
i\geq 0
\end{equation}%
which follows from the inequality $\frac{\alpha +\gamma }{\beta +\delta }%
\leq \max (\frac{\alpha }{\beta },\frac{\gamma }{\delta }),$ $\alpha ,\gamma
\geq 0,$ $\beta ,\delta >0$ and equation $(25)$ so that

\begin{equation}
\left( \frac{\log \phi _{+}^{\prime }}{\log \sigma _{+}},\ \frac{\log \phi
_{-}^{\prime }}{\log \sigma _{-}}\right) \geq 1.
\end{equation}%
Setting $\sigma ^{-1}\phi _{+}^{\prime }=\{\frac{X_{+}}{x}\}^{i},\ $ $\sigma
\phi _{-}^{\prime }=\{\frac{X_{-}}{x}\}^{i}$ and $\sigma =x^{-v(\tilde{x})}$
the multiplicative neighbours of $x$ are obtained as
\begin{equation}
X_{\pm }=x.\ x^{\mp v(\tilde{x})}
\end{equation}%
The Cantor function $\phi (\tilde{x})$ over the infinitesimals $\tilde{x}$
is thus given by
\begin{equation}
\phi (\tilde{x})=\log _{ x^{-1}}\frac{X (\tilde{x})}{x}=v(\tilde{x})
\end{equation}%
thereby retrieving the variability of $\phi $ relative to $v$ trivially viz
: $d\phi =dv.$ 

We note that this again explains explicitly the removal of derivative discontinuities as encoded in eq(24) in the present formalism. The divergence of either the left or right derivative at an $x\in C$, that arises due to the divergence of $(r/p)^k, \ k\rightarrow \infty$, is smoothed out in the logarithmic variables that replace the ordinary limiting variables as in eqns (25) and (26), which, in fact, correspond to eq(22).    
We conclude that the multiplicative non-archimedean structure given by $(29)$
induces a smoothening effect on the discontinuity of $\phi ^{\prime }(x)$ in
the usual topology.

\end{document}